\documentclass[reqno]{amsart}     
\usepackage{amsmath,amssymb}    
\usepackage{mathrsfs}    
\usepackage{color} 
\usepackage[backref,colorlinks]{hyperref}  
\usepackage[english,algoruled,vlined,linesnumbered,norelsize]{algorithm2e}

\newtheorem{theorem}                   {Theorem} 
 
\newtheorem{lemma}           [theorem] {Lemma}  
   
\newtheorem{corollary}       [theorem] {Corollary}   
\newtheorem{cor}             [theorem] {Corollary}

\theoremstyle{remark}
\newtheorem{remark}          [theorem] {Remark}

\newcommand{\eps}{\varepsilon}
\newcommand{\cH}{\mathcal{H}}
\newcommand{\cC}{\mathcal{C}}
\newcommand{\DFS}{\mathrm{DFS}}
\newcommand{\EE}{\mathbb{E}}

\newcommand{\Nat}{\mathbb{N}}
\newcommand{\Bi}{\mathrm{Bi}}

\newcommand{\Ecomp}{E^{(1)}{(t)}}
\newcommand{\dstart}{d_L^{(1)}(t)}
\newcommand{\Dstart}{D_\ell^{(1)}(t)}
\newcommand{\djump}{d_L^{(2)}(t)}
\newcommand{\Djump}{D_\ell^{(2)}(t)}
\newcommand{\Ejump}{E^{(2)}{(t)}}
\newcommand{\dbranch}{d_L^{(3)}(t)}
\newcommand{\Dbranch}{D_\ell^{(3)}(t)}
\newcommand{\Ebranch}{E^{(3)}(t)}

\newcommand{\dtotal}{d_L(G_j(t))}
\newcommand{\Dtotal}{\Delta_{\ell}(G_j(t))}
\newcommand{\Etotal}{E{(t)}}

\newcommand{\Dfinaltotal}{\Delta_{\ell}(G_j(\alpha n^k))}

\newcommand{\goodjumps}[2]{A^{(2)}_{#1}(#2)}
\newcommand{\allgoodjumps}{A^{(2)}}
\newcommand{\goodbranchings}[2]{A^{(3)}_{#1}(#2)}
\newcommand{\allgoodbranchings}{A^{(3)}}
\newcommand{\goodstarts}[1]{A^{(1)}(#1)}
\newcommand{\allgoodstarts}{A^{(1)}}

\newcommand{\allgoodevents}{A^{(\mbox{\rm all})}}

\newcommand{\brconst}{C^\dagger}

\begin{document}
\title[Largest components in random hypergraphs]{Largest components in random hypergraphs}
\thanks{The first and third authors were supported by short visit grants~5639 and~5472 respectively from 
the European Science Foundation (ESF) within the 
``Random Geometry of Large Interacting Systems and Statistical Physics'' (RGLIS) program.\\
The second author is supported by Austrian Science Fund (FWF): P26826, W1230, Doctoral Program ``Discrete Mathematics''.}

\author[O.~Cooley]{Oliver Cooley}
\address{Institute of Optimization and Discrete Mathematics, Graz University of Technology, Steyrergasse 30, 8010 Graz, Austria}
\email{cooley@math.tugraz.at}
\author[M.~Kang]{Mihyun Kang}
\address{Institute of Optimization and Discrete Mathematics, Graz University of Technology, Steyrergasse 30, 8010 Graz, Austria}
\email{kang@math.tugraz.at}
\author[Y.~Person]{Yury Person}
\address{Goethe-Universit\"at, Institute of Mathematics, Robert-Mayer-Str. 10, 60325 Frankfurt, Germany}
\email{person@math.uni-frankfurt.de}

\date{\today}

\begin{abstract}
In this paper we consider $j$-tuple-connected components in random $k$-uniform hypergraphs (the $j$-tuple-connectedness relation can be defined by letting two $j$-sets be connected if they lie in a common edge and consider the transitive closure; the case $j=1$ corresponds to the common notion of vertex-connectedness). We determine that the existence of a $j$-tuple-connected component containing $\Theta (n^j)$ $j$-sets in random $k$-uniform hypergraphs undergoes a phase transition and show that the threshold occurs at edge probability $\tfrac{(k-j)!}{\binom{k}{j}-1}n^{j-k}$. Our proof extends the recent short proof for the graph case by Krivelevich and Sudakov which makes use of a depth-first search to reveal the edges of a random graph.

Our main original contribution is a \emph{bounded degree lemma}, which controls the structure of the component grown in the search process.

\vspace{0.5cm}


\end{abstract}

\maketitle
\section{Introduction}

\subsection{Phase transition in random graphs}
The Erd\H{o}s-R\'enyi random graph~\cite{ErdRen60} $G(n,p)$(resp.\ $G(n,M)$) 
is one of the most intensely studied in the theory of random graphs. 
It is well-known and has been studied in great detail (see e.g.~\cite{Bolbook,JLRbook}) 
how the structure of the components changes as $p$ (resp.\ $M$) grows.
In the seminal paper~\cite{ErdRen60} entitled ``On the evolution of random graphs'' Erd{\H o}s and  R\'enyi discovered among other things that the Erd{\H o}s-R\'enyi random graph undergoes a drastic change of the size and structure of 
largest components, which happens when the number of edges is around $n/2$. In terms of the binomial model $G(n,p)$, this phenomenon can be stated as follows. Consider $G(n,p)$ with $p=c/n$ for a constant $c>0$. If $c<1$, then 
asymptotically almost surely (a.a.s.\ for short, meaning with probability tending to one as $n$ tends to $\infty$) all the components in $G(n,p)$ have $O(\log n)$ vertices, whereas if $c>1$, then a.a.s.\ there is a unique component with $\rho n+ o(n)$ vertices, the so-called giant component, where $\rho$ is the unique positive solution of the  equation
\begin{equation*}
1-\rho=\exp(-c\rho).
\end{equation*} 
In 1984, Bollob\'as~\cite{Bol84} made a breakthrough in the study of the so-called critical phenomenon associated with the phase transition 
by studying the case that $c\to1$ in more detail. His result was improved by {\L}uczak in 1990~\cite{Luczak90}.  Let $\lambda$ be such that 
\begin{equation}
p = \frac 1n + \frac\lambda{n^{4/3}}.\label{eq:lambda}
\end{equation}
If $\lambda\to-\infty$,  then a.a.s.\ all the components have order $o(n^{2/3})$. If $\lambda\to +\infty$, then there is a.a.s.\ a unique component of order $\gg n^{2/3}$, while all other components have order $o(n^{2/3})$. If $\lambda$ is a constant, then the size of the largest component is $\Theta (n^{2/3})$.

\subsection{Phase transition in random hypergraphs} 
A $k$-uniform hypergraph $H$ is the tuple $(V,E)$, where $V$ is the vertex set of $H$ 
and $E$ is its edge set with $E\subseteq \tbinom{V}{k}$. The random $k$-uniform hypergraph 
$H^k(n,p)$ is defined similarly to $G(n,p)$: each of the $\tbinom{n}{k}$ 
possible edges is included independently of the others with probability $p$.  

Similar phase transition phenomena were discovered in random hypergraphs.  In particular, a straightforward generalisation of the giant component was studied in~\cite{SPS85,KarLuc02,BCOK10,BCOK13}, where the following concept of ``component'' was studied: 
two vertices $u$ and $v$ are connected in a $k$-uniform hypergraph if there is a sequence 
of edges $h_0$,\ldots, $h_\ell$ such that $u\in h_0$ and $v\in h_\ell$ 
and $h_i\cap h_{i+1}\neq\emptyset$.

The threshold for  $H^k(n,p)$ was first determined by Schmidt-Pruzan and Shamir in~\cite{SPS85}. 
More precisely, let $p=c/\binom{n-1}{k-1}$. If $c<(k-1)^{-1}-\eps$ for an arbitrarily small but fixed $\epsilon>0$, then a.a.s.\  the number of vertices of the largest component is $O(\log n)$. But, if $c>(k-1)^{-1}+\eps$, then a.a.s.\ there is a unique component containing a linear number of vertices,
which is called the \emph{giant component}; more precisely, the number of vertices of the giant component is $\rho n+o(n)$, where $\rho$ is the unique positive solution to the equation
\begin{equation*}
1-\rho=\exp(c((1-\rho)^{k-1}-1)).
\end{equation*}
Karo\'nski and \L uczak~\cite{KarLuc02} studied the phase transition in the \emph{early supercritical phase}, when
$\tbinom{n}kp=\frac{n}{k(k-1)}+o(n^{2/3}(\log n/\log\log n)^{1/3})$, proving a local limit theorem for the number
of vertices in the largest component. Ravelomanana and Rijamamy~\cite{RR06} extended the range to
$\tbinom{n}kp=\frac{n}{k(k-1)}+o(n^{7/9})$, although only giving the expected size of the largest component and
not its distribution. Behrisch, Coja-Oghlan and Kang~\cite{BCOK10} established central and local limit theorems for
the number of vertices in the largest component of $H^k(n,p)$  with edge probability $p(k-1)\tbinom{n-1}{k-1}>1+\epsilon$
for an arbitrarily small but fixed $\epsilon>0$. Bollob\'as and Riordan~\cite{BR12} subsequently proved that the distribution
of the number of vertices in the largest component tends to a normal distribution whenever $\eps = \omega (n^{-1/3})$.

While in the graph case two vertices are connected if there is a path (or walk) between them, in hypergraphs the notion of a path (or walk) is ambiguous and in fact there are several possible definitions. 
An \emph{$s$-tight path} of length $m$ in a $k$-uniform hypergraph $H$ is a sequence $e_0$, \ldots, $e_{m-1}$ such that $e_i=\{v_{i(k-s)+1},\ldots v_{i(k-s)+k}\}$ for some distinct vertices $v_j$. In the case $s=1$ we call an $s$-tight path \emph{loose}, and for $s=k-1$ simply \emph{tight} path. 

Note that when $p=\tfrac{(k-2)!}{n^{k-1}}$,  the edges in $H^k(n,p)$ typically intersect in at most one vertex, thus, ``morally'' if two vertices are connected then they are connected by loose paths (mostly). The result of Schmidt-Pruzan and Shamir in~\cite{SPS85} (mentioned above) can be restated as follows.
\begin{theorem}[\cite{SPS85}]\label{thm:vtx_conn}
 Let $k\ge 2$ and $\eps>0$ be given. Then a.a.s.\ the (vertex) size of the giant component in the random 
 $k$-uniform hypergraph $H^k(n,p)$ is $\Omega(\eps n)$ if $p=(1+\eps)\tfrac{(k-2)!}{n^{k-1}}$ and 
 $O(\tfrac{\log n}{\eps^2})$ if $p=(1-\eps)\tfrac{(k-2)!}{n^{k-1}}$. 
\end{theorem}

In this paper we study the following notion of $j$-tuple-connectivity in $k$-uniform hypergraphs, which 
generalises the notion mentioned above (if $j=1$ we simply speak about vertex-connectivity). 
We say that two $j$-sets (tuples) $J_0$ and $J_n$ ($j\in\{1,\ldots,k-1\}$) 
are \emph{$j$-tuple-connected} in the $k$-uniform hypergraph $H$ if there is an alternating sequence of 
$j$ and $k$-element subsets of $V(H)\colon$ 
$J_0,h_0,J_1,h_1,\ldots, J_n$ such that  $J_i\cup J_{i+1} \subseteq h_i$ and $h_i \in E(H)$. 
The components  then consist of   $j$-element subsets of the vertex set of $H$.  
Again, one might wonder when a $j$-tuple-connected giant component of size (i.e.\ number of $j$-sets) $\Theta(n^j)$ emerges in the random $k$-uniform hypergraph. 

For the rest of the paper we will regard $k$ and $j$ as fixed constants. In particular, this means that any parameter which is a function only of $k$ and $j$ is also a fixed constant.

\subsection{Intuition: where to locate the thresholds?}
The intuition (as in the case of random graphs) comes from 
the branching processes which can be  described for general 
case of $j$-tuple-connectivity as follows. Initially we start with 
a single $j$-element set $J_0$. We expect that there are 
$m:=p\tbinom{n-j}{k-j}\approx pn^{k-j}/(k-j)!$ edges $e_1$,
\ldots, $e_m$ containing $J_0$ in
$H^k(n,p)$. The range of $p$ is typically such that these edges 
intersect pairwise only in $J_0$, which leads to $\tbinom{k}{j}-1$ 
offspring for each edge $e_i$. From the theory of branching processes, 
the process survives with positive probability for indefinite time if
$\left(\tbinom{k}{j}-1\right)m>1$
(if everything is independent and binomially distributed). This suggests that the threshold should 
be 
\begin{equation}\label{eq:threshold}
p_{k,j}=p_{k,j}(n):=\frac{(k-j)!}{\binom{k}{j}-1}n^{j-k}.
\end{equation}
For $j=1$ we obtain $p=\tfrac{(k-2)!}{n^{k-1}}$, which is exactly the threshold, 
see~\cite{ErdRen60, SPS85}. For $j=k-1$ we obtain the conjectured threshold for 
the  $(k-1)$-tuple-connectivity to be $\tfrac{1}{(k-1)n}$. 
Our main theorem shows that $p_{k,j}$ is the correct threshold for all $k,j$, as was suggested recently by 
Bollob\'as and Riordan in~\cite{BR12}.

Our approach builds on the recent proof strategy of Krivelevich and Sudakov~\cite{KrivSud13} who used the 
depth-first search algorithm in graphs  to give a simple and short proof of the phase transition in $G(n,p)$.
More precisely, we first adapt their proof strategy for $j=1$, thus deriving an alternative proof of Theorem~\ref{thm:vtx_conn}.
Moreover, the approach via depth-first search allows us to study the largest component in the supercritical phase, i.e.\ 
when $p=(1+\eps)\tfrac{(k-2)!}{n^{k-1}}$ with $\eps=\eps(n)\gg n^{-1/3}$, which gives the lower bound $\Omega(\eps n)$ for the size of the largest component in random hypergraphs. 
This range of $\eps$ matches that considered by Bollob\'as and Riordan in~\cite{BR12} and is essentially best possible.

Then we turn to the case of general $k,j$, which requires some additional work, most notably Lemma~\ref{lem:maxdeg}.
We obtain the following theorem thus confirming the threshold for $p_{k,j}$ 
mentioned above. By $\omega(f(n))$ we denote any function $g(n)$ such that $g(n)/f(n)\to\infty$ as $n\to\infty$.
\begin{theorem}\label{thm:pair_conn}
 Let $\eps= \eps(n)>0 $ and $1\leq j \leq k-1$ be given. Then a.a.s.\ the size of the $j$-tuple-connected largest component in the 
 random $k$-uniform hypergraph $H^k(n,p)$ is $O(\eps^{-2}\log n)$ if $p=(1-\eps)p_{k,j}$. 

Let $\delta\in(0,1)$ be any constant. If furthermore $\eps = \omega(n^{\delta-1}+n^{-j/3})$, 
then a.a.s.\ the largest $j$-tuple-connected component in $H^k(n,p)$ has size $\Omega(\eps n^j)$ if $p=(1+\eps)p_{k,j}$.
\end{theorem}

Note in particular that Theorem~\ref{thm:vtx_conn} is an immediate corollary, although we will prove Theorem~\ref{thm:vtx_conn} first, since the proof of the special case is substantially simpler.

Note also that there is no lower bound on the size of $\eps$ in the first part of the theorem. However, for very small $\eps$, the bound on the largest component is not best possible, and may even be greater than $n^j$, and therefore useless as a bound. We discuss the critical window in more detail in Section~\ref{sec:conclusion}.

While preparing this paper we discovered that independently Lu and Peng~\cite{LP14} have claimed to have a proof of a similar result, although only for constant $\eps$.

Our main contribution to the proof of Theorem~\ref{thm:pair_conn} is Lemma~\ref{lem:maxdeg}, which will be formally stated in Section~\ref{sec:post-pt}. Briefly, it states that for some $\alpha \in (0,1)$ which is a function of $\eps$, with high probability the set of $j$-sets which have been discovered by time $\alpha n^k$ is ``smooth'' in the following sense: for any $1\le \ell \le j-1$, any $\ell$-set is contained in $O(\alpha n^{j-\ell})$ such $j$-sets. (This is best possible up to a constant factor.) 

In the proof of Theorem~\ref{thm:pair_conn} (and the special case $j=1$, which is Theorem~\ref{thm:vertex_one}), for the case $p=(1+\eps)p_{k,j}$ we will implicitly assume that $\eps$ is less than some small constant, say $\eps_0$, which is dependent on $k,j$. This is permissible since if $\eps > \eps_0$, then our aim is prove that there is a component of size $\Omega (\eps n^j) =\Omega (\eps_0 n^j) = \Omega (n^j)$, thus the result for $\eps>\eps_0$ is implied by the result for $\eps = \eps_0$. 

\subsection{Motivation from random simplicial complexes}
A parallel development was initiated by Linial and Meshulam who studied homological connectivity of 
 random simplicial complexes~\cite{LinMes06}. Furthermore, 
motivated by finding thresholds for various algebraic notions of cycles in $H^k(n,p)$, 
the questions such  as collapsibility and vanishing of the top homology have been investigated in~\cite{ALLM12,AL13,AL13b}. 
A $k$-uniform hypergraph $H$ is collapsible if one can delete from $H$ all of its edges 
one by one, such that in each step we remove some edge $e$ containing a $(k-1)$-element 
set $J$ if $e$  doesn't intersect any other edge in $J$. 
It has been shown in~\cite{ALLM12} that the first emerging cycle in the $(k-1)$th homology group of  $H^k(n,p)$  is 
either $K^{k}_{k+1}$ or contains $\Omega (n^{k-1})$ edges. 
Our  Theorem~\ref{thm:pair_conn} in the case $j=k-1$ may be seen as the study of the acyclic case of $H^k(n,p)$
 where we have a sharp threshold for the emergence of a tightly connected ``hypertree'' with $\Theta(n^{k-1})$ edges covering all vertices.

\section{Exploration of random hypergraphs via depth-first search}

Usually we denote the $n$-vertex set of a hypergraph by $[n]:=\{1,2,\ldots,n\}$. We also denote by $e(\cH)$ the order of the edge set $E(\cH)$.

\subsection{Exploration algorithm in hypergraphs}
Now we introduce the depth-first search algorithm (DFS) for hypergraphs. 
We are given as input two hypergraphs $H$ and $\cH$ on the same vertex set $V$ with $E(H)\subseteq E(\cH)$, 
and we would like to discover all edges of $H$ by \emph{querying} the edges of $\cH$ whether they belong to $H$. 
We will choose vertices and edges to query according to some linear orderings (see Algorithm~\ref{alg:DFS}).

There are three types of vertices: \emph{neutral}, \emph{active} and \emph{explored} (we borrow the terminology 
from~\cite{NPer10}). Additionally vertices that are active or explored (i.e. any non-neutral vertices) are called \emph{discovered}. Initially, all vertices are neutral. We start our exploration from the smallest neutral vertex, which we mark as active. Departing from some (currently considered) active vertex $v$ we try to 
discover the smallest edge $e$ of $\cH$ (we also say that we \emph{query} $e$) such that $e\in H$ and $v\in e$ containing at least one neutral vertex (but no explored vertices). 
Once such an edge is found we mark all neutral vertices in this edge active and start the same query process from the vertex which has been marked active last. If no edge $e$ could be found, we mark $v$ as explored and start the same querying process from the next active vertex. If no vertices are active then we have discovered some component completely, and we proceed as in the beginning of the depth-first search. Finally, once all vertices are discovered, we query all unqueried edges. Notice however that at the moment when all vertices are discovered, we know the vertex components of $H$.

Below is the complete description of the algorithm.
\begin{algorithm}[h]
  \caption{Hypergraph exploration $\DFS$}\label{alg:DFS}
    \KwIn{$\cH$, $H=(V,E)\subseteq \cH$ -- $k$-uniform hypergraphs.\linebreak $\sigma$ -- linear ordering of $E(\cH)$.\linebreak $\tau$ -- linear ordering of $V(\cH)$.}
    \KwOut{$\cC$ -- set of  vertex-connected components of $H$}
\BlankLine    
    $\cC:=\emptyset$\;
    let $S$ be empty stack\;
    \Repeat{all vertices are explored}{
    let $x$ be the smallest neutral vertex in the ordering $\tau$\;
    mark $x$ as active\;
    add $x$ to $S$\;
    \While{$S\neq\emptyset$}{
      Let $x$ be the top vertex of $S$\;
      \eIf{$\exists$ the smallest unqueried edge 
       $e$ of $\cH$ such that $x\in e$ and $e$ contains a neutral vertex}{ 
       \If{$e\in E(H)$ \bf{(``query $e$'')}  } {add all neutral vertices of $e$ in ascending order to top of $S$\; mark these vertices as active\;}}
     {remove $x$ from $S$\; mark $x$ as explored\;} 
      }
     let $C$ be the set of vertices explored in the while-loop above\;
     $\cC:=\cC\cup \{C\}$\;
    } 
    query all remaining edges of $\cH$ in the ascending order\;
\end{algorithm}

\subsection{Coupling}\label{sec:coupling}
 Let $H:=\cH_p$  denote the random subhypergraph of $\cH$ where every edge of $\cH$ is chosen independently 
of the other edges with probability $p$. Further let $(X_i)_{i\in[e(\cH)]}$ be 
a sequence of $e(\cH)$ i.i.d.\ Bernoulli random variables with mean $p$. 
We can associate with the $i$th query the random variable $X_i$, meaning that 
if $X_i=1$ then the queried edge is in $H$ and otherwise not. 
Once we have fixed the orderings of the vertices and edges and the values of the $X_i$, Algorithm~\ref{alg:DFS} is 
a deterministic one and it queries every edge of $\cH$ exactly once, thus every $\{0,1\}$-sequence
of length $e(\cH)$ corresponds to a unique subgraph of $\cH$. 

In this way, for given fixed orderings $\sigma$ and $\tau$ as in Algorithm~\ref{alg:DFS}, we couple the $X_i$'s with 
$\cH_p$. For technical reasons which are not needed in the case of vertex-connectivity we let the choice of $\tau$ be uniformly at random independently of $\sigma$.  In fact, we will only need that $\tau$ is chosen randomly in this way at one point in the paper (Lemma~\ref{lem:goodstarts}) --  an arbitrary ordering $\tau$ would work almost as well, but would lead to some additional technical difficulties in the range when $\eps$ is very small.

\begin{remark} Note that the ordering $\tau$ is only used when starting a new component to determine which vertex we will continue exploring -- the rest of the algorithm is independent of $\tau$. It is easy to see that this is equivalent to choosing a neutral vertex uniformly at random from which to continue exploring. It is this interpretation that we will consider in Lemma~\ref{lem:goodstarts}.
\end{remark}

We will show the existence of a large component (Theorem~\ref{thm:pair_conn}) by proving that Algorithm~\ref{alg:DFS}
(or one of its relatives, which will be defined later), will find a.a.s.\ after $\alpha n^k$ queries a large $j$-tuple-connected component in $H^k(n,p)$ for 
appropriate small $\alpha$.

In the case of vertex-connectivity our $\cH=\left([n],\tbinom{[n]}{k}\right)$ is the complete $k$-uniform 
hypergraph and $H$ is the random $k$-uniform hypergraph $H^k(n,p)$. 

In order to study $j$-tuple-connectivity we could alter Algorithm~\ref{alg:DFS}, in that 
we visit $j$-element sets of vertices instead of single vertices. Instead of this,
 we define the $\tbinom{k}{j}$-uniform hypergraph  $\cH$ as follows: 
 The vertex set of $\cH$ is $\tbinom{[n]}{j}$ and the edges are those $\tbinom{k}{j}$-element 
 subsets of $V(\cH)$ that consist of all $j$-element subsets of some $k$-element set from $[n]$. 
Thus, we have reduced a question about $j$-tuple-connectivity in a $k$-uniform hypergraph to one 
about vertex-connectivity in an appropriately defined auxiliary hypergraph $H\subseteq \cH$. 
In the following we will analyse Algorithm~\ref{alg:DFS} when applied to $\cH$ and $H=\cH_p$.

At several points in this paper we will want to calculate an upper bound on the number of edges found in some subset of the DFS process, e.g.\ the set of queried edges that contain a given vertex $v$ 
or, generally, a given $\ell$-set $L$. 
It will often be convenient to simplify such situations by allowing some additional queries which are not actually made within this subset (or possibly within the DFS process at all). Formally, we couple the subset of the DFS with a number of dummy variables which are also i.i.d.\ Bernoulli random variables with probability $p$, and which mimic these additional queries. Then the number of `1's in the subset we consider is certainly at most the number of `1's in the subset together with the dummy variables. In what follows we shall therefore assume the existence of these extra queries without mentioning the formal interpretation.

\subsection{Chernoff bounds}
We will use the following version of the Chernoff bound from~\cite[Theorem~2.1]{JLRbook}.
\begin{theorem}\label{thm:chernoff}
Let $X$ be the sum of $t$ i.i.d.\ Bernoulli random variables with mean $p$, then for $a \ge 0$,
\begin{align*}
   \Pr\left[X\ge \EE(X)+a\right] & \le \exp\left(-\frac{a^2}{2(tp+a/3)}\right)\\
   \Pr\left[X\le \EE(X)-a\right] & \le \exp\left(-\frac{a^2}{2tp}\right).
\end{align*}
\end{theorem}


\section{Before phase transition}\label{sec:pre-pt}
First we prove that when $p$ is not too large, all the components in a generalised random hypergraph on $N$ vertices have size $O(\tfrac{\log N}{\eps^2})$. 
We start with an auxiliary lemma (cf.~\cite[Lemma~1]{KrivSud13}).
\begin{lemma}\label{lem:zero}
 Let $M \in \mathbb{N}$, $\eps\in(0,1)$, $c\in \mathbb{R}$ and let $(X_i)_{i\in[M]}$ be i.i.d.\ random Bernoulli variables with mean $p$. 
 If $p\le\tfrac{1-\eps}{c}$ and $t\ge\tfrac{9c\log M}{\eps^2}$ 
 then with probability at least $1-M\exp(\tfrac{-\eps^2t}{3c}) \ge 1-1/M^2$, the sum of $X_i$'s within any subinterval of 
 length $t$  of the interval $[M]$ is less than  $\tfrac{t}{c}-1$.
\end{lemma}
\begin{proof}
 For $p=\tfrac{1-\eps}{c}$ we have $\EE (\sum_{i=t_0}^{t_0+t-1} X_i)=(1-\eps)t/c$.
 We apply Theorem~\ref{thm:chernoff} 
 to bound the probability that the sum within a fixed interval
 of length $t$ is at least $\tfrac{t}{c}-1= \EE (\sum_{i=t_0}^{t_0+t-1} X_i) + \eps\tfrac{t}{c} -1$:
\[
   \Pr\left[\sum_{i=t_0}^{t_0+t-1} X_i> \tfrac{t}{c}-1\right]
   \le \exp\left(-\frac{(\eps\tfrac{t}{c} -1)^2}{2((1-\eps)t/c+\eps t/(3c))}\right)
  <\exp(\tfrac{-\eps^2t}{3c})\le1/M^{3}.
\]  
The union bound over all possible intervals gives the claim.
\end{proof}

\begin{theorem}\label{thm:zero_a}
 Let $\cH$ be an $\ell$-uniform hypergraph on $N$ vertices with maximum degree $\Delta$. 
 Then for $p\le \tfrac{1-\eps}{(\ell-1)\Delta}$, the random hypergraph $\cH_p$ has a.a.s.\ (as $N\rightarrow \infty$)
 only vertex-connected components of size at most $\tfrac{9(\ell-1)\log (\Delta N/\ell)}{\eps^2}$.
\end{theorem}
\begin{proof}
We couple $\cH_p$ with a sequence $(X_i)_{i\in[e(\cH)]}$ of i.i.d.\ Bernoulli variables with mean $p$, 
as described in Section~\ref{sec:coupling} (observe: $e(\cH)\le\Delta N/\ell$).  We run Algorithm~\ref{alg:DFS} to explore the hypergraph 
$\cH_p$. 

If there exists a component $C\subseteq V(\cH)$ of size at least 
$\tfrac{9(\ell-1)\log (\Delta N/\ell)}{\eps^2}$, then it is found during some while-loop. 
Since $C$ is 
connected and its vertices are explored, 
we have found in this while-loop at least $(|C|-1)/(\ell-1)$ edges in the
component $C$ (during while-loop, each time we find an edge $e\in H$ we gain
  at most $\ell-1$ new active vertices). The number of queries during this while-loop 
 is  at most $|C|\Delta\ge \tfrac{9(\ell-1)\Delta\log (\Delta N/\ell)}{\eps^2}$. Let us assume for an upper bound that the number of queries is exactly $|C|\Delta$. But then by Lemma~\ref{lem:zero}, we have a.a.s.\ that the number of $X_i$'s 
which are $1$ and thus the number of edges discovered (so far) is  less than  
$\tfrac{|C|}{\ell-1}-1$, contradicting the fact that at least $(|C|-1)/(\ell-1)$ 
edges have been discovered for \emph{every} interval of length $|C|\Delta$, as explained above.
\end{proof}

From Theorem~\ref{thm:zero_a} we immediately obtain the cases of Theorems~\ref{thm:vtx_conn} 
and~\ref{thm:pair_conn} when $p\le (1-\eps)p_{k,j}$. 

\begin{corollary}\label{cor:zero}
  Let $\eps>0$, $k,j\in\Nat$ with $k>j$ be given. If $p\le (1-\eps)\tfrac{(k-j)!}{\binom{k}{j}-1}n^{j-k}$, then 
  a.a.s.\ the $j$-tuple-connected components of the random $k$-uniform hypergraph 
  $H^k(n,p)$ have size $O(\eps^{-2}\log n)$.
\end{corollary}
\begin{proof}
 We define $\cH$ as follows: the vertex set $V(\cH):=\tbinom{[n]}{j}$ and the set of edges 
 \[
 E(\cH):=\left\{\tbinom{U}{j}\colon U\in \tbinom{[n]}{k} \right\}.
 \]
 Thus, every $S\in V(\cH)$ has degree $\deg_{\cH}(S)=\tbinom{n-j}{k-j}$, implying that if \linebreak[4]
 $p\le (1-\eps)\frac{(k-j)!n^{j-k}}{\binom{k}{j}-1}\le (1-\eps/2) \frac{1}{\left(\binom{k}{j}-1\right)\binom{n}{k-j}}$ (for $n$ large enough),
 then $\cH_p$  has components of size at most
\begin{align*}
\frac{9}{(\eps/2)^2}\left(\binom{k}{j}-1\right)\log \left( \frac{\binom{n-j}{k-j}\binom{n}{j}}{\binom{k}{j}-1} \right)
& \leq \frac{36}{\eps^2}\binom{k}{j}\log (n^k)\\
& =O(\eps^{-2}\log n).
\end{align*}

Therefore, the random hypergraph $H^k(n,p)$ has, for the same $p$, $j$-tuple-connected components of size at most
$O(\eps^{-2}\log n)$ a.a.s.. Thus, the assertions of Theorems~\ref{thm:vtx_conn} 
and~\ref{thm:pair_conn} follow when $p\le (1-\eps)p_{k,j}$.
\end{proof}

\section{After phase transition}\label{sec:post-pt}
\subsection{Algorithm 2}
For the regime when $p\ge (1+\eps)p_{k,j}$ we slightly alter our algorithm in that 
in the main \emph{if}-condition during the while-loop we only consider those unqueried edges 
$e\in E(\cH)$ such that $e\setminus\{x\}$ consists of only \emph{neutral} vertices. Thus, 
when an edge in $H$ (during some while-loop) is found, we get $\binom{k}{j}-1$ new active vertices.  
We refer 
to this algorithm as \emph{Algorithm~2}. Observe that in this case we may not fully discover the $j$-tuple-connected 
components, but if we find a sufficiently large partial component in this way then this clearly gives a lower bound 
on the size of the largest component.

\subsection{Vertex-connectivity}
First we look at the case of the vertex-connectivity in the random $k$-uniform hypergraph. As mentioned above,
Algorithm~2 gives a  lower bound on the size of the largest component (indeed, it constructs a ``hypertree'' in $H^k(n,p)$ of this size). Furthermore, since the expected number of edges sharing two particular vertices 
 in $H^k(n,p)$ for 
$p=(1+\eps)\tfrac{(k-2)!}{n^{k-1}}$ is $O(1)$, we expect that even after removing all such pairs of edges, 
most of the largest component remains connected via loose paths.

We need the following auxiliary lemma.

\begin{lemma}\label{lem:strongconcentration}
Let $j<k\in\Nat$ and let $X_1$,\ldots $X_{\alpha n^k}$ be i.i.d.\ Bernoulli random variables with parameter $p\le k! n^{j-k}$. 
Suppose $\alpha=\alpha(n)$ is such that $\alpha^3 n^j \rightarrow \infty$. 
Then with high probability for every $1 \le t \le  \alpha n^k$ we have
\[
   \left|\sum_{i=1}^t X_i - pt\right|\le \alpha^2 n^j.
\]
\end{lemma}

Note that the concentration given by this lemma is useless for very small $t$. However, we will only need to apply it for $t=\Theta(\alpha n^k)$, where it gives a better concentration than that which would be given by applying a Chernoff bound and a union bound over all $t$.

To prove this lemma, we will need the following martingale result, an asymmetric version of the Hoeffding inequality proved by Bohman~\cite{Bo09}.  

\begin{lemma}[Lemmas~6 and~7 from~\cite{Bo09}]\label{lem:martingale}
Suppose $0=Y_0,Y_1,\ldots,Y_m$ is a martingale in which $-c \le Y_{i} - Y_{i-1} \le C$ for all $1\le i \le m$ and some real numbers $c,C >0$ with $c \le C/10$. Then for every $0<a<cm$,
\[
\Pr (|Y_m| \ge a) \le 2\exp \left(\frac{-a^2}{3cCm}\right) .
\]
\end{lemma}

\begin{proof}[Proof of Lemma~\ref{lem:strongconcentration}]
Let us define a martingale $Y_0,Y_1,\ldots,Y_{\alpha n^k}$ as follows:
\begin{align*}
Y_0 & :=0\\
Y_{i+1} & :=
\begin{cases}
Y_i + X_{i+1}-p & \mbox{if } |Y_i| \le \alpha^2 n^j ; \\
Y_i & \mbox{otherwise.}
\end{cases}
\end{align*}
Note that this may be seen as a martingale with a stopping time, where the stopping condition is $|Y_i| > \alpha^2 n^j$. It is easy to check that this is indeed a martingale. Furthermore, we have $-p \le Y_{i+1}-Y_i \le 1-p \le 1$. Therefore by Lemma~\ref{lem:martingale} we have
\[
\Pr \left( |Y_{\alpha n^k}| > \alpha^2 n^j\right) \le 2\exp \left( -\frac{(\alpha^2 n^j)^2}{3p\alpha n^k}  \right)
 \le 2\exp \left( -\alpha^3 n^j/(3 k!) \right)
 = o(1).
\]
Furthermore, note that the conclusion of Lemma~\ref{lem:strongconcentration} holds if and only if $|Y_{\alpha n^k}| \le \alpha^2 n^j$, and therefore the above calculation proves the lemma.
\end{proof}

With the above lemma to hand we follow the lines of~\cite[Theorem~2]{KrivSud13} to show

\begin{theorem}\label{thm:vertex_one}
 Let $k\in\Nat$, $k\ge2$ and let $\eps= \eps (n)$ be a function satisfying $\eps = \omega (n^{-1/3})$. If $p=(1+\eps)\tfrac{(k-2)!}{n^{k-1}}$, then a.a.s.\
 the random hypergraph $H^k(n,p)$ contains $\Omega(\eps n)$ vertices that are pairwise connected by loose paths.

\noindent In particular, $H^k(n,p)$ has a vertex-connected component of size $\Omega(\eps n)$.
\end{theorem}
\begin{proof}
We consider a sequence of i.i.d.\ Bernoulli random variables coupled with $H^k(n,p)$, as explained in Section~\ref{sec:coupling}. 
Our $\cH$ is the complete $k$-uniform hypergraph $K^k_n$ with $n$ vertices. 

We choose $\alpha:=\tfrac{\eps}{8k!}$ with foresight.

We shall claim  that between the query $\tfrac{\alpha}{2} n^k$ and the query $\alpha n^k$ the stack $S$ of active vertices a.a.s.\ hasn't been empty, meaning that 
during this time Algorithm~2 is discovering a single (large) component. 
This, together with Lemma~\ref{lem:strongconcentration}, implies that a.a.s.\ the number of $X_i$'s that are answered as $1$ between 
the query $\tfrac{\alpha}{2}n^k$ and $\alpha n^k$ is at least 
\[
  (p\alpha n^k - \alpha^2 n) - (p\alpha n^k/2 + \alpha^2 n) \ge (k-2)!\alpha n/2 - 2\alpha^2 n \ge \frac{\eps}{16k^2}n,
\]
which yields the assertion of Theorem~\ref{thm:vertex_one}, 
since there are still some unexplored vertices, and therefore Algorithm~2 has found at 
least $\tfrac{\eps}{16 k^2}n$ edges in some component and each such edge makes $k-1$ previously neutral vertices active, which results 
in a component of size $\Omega(\eps n)$.

To prove the claim let us assume that after some $t$ queries where \linebreak[4]
$t\in\{\tfrac{\alpha}{2}n^k,\ldots,\alpha n^k\}$, the 
stack $S$ is empty. By Lemma~\ref{lem:strongconcentration}, a.a.s.\ Algorithm~2 has discovered 
$
pt\pm \alpha^2 n
$
 edges in $H^k(n,p)$. Since with each explored edge, $k-1$ vertices become active, and after emptying the stack $S$ 
all active vertices are explored, this implies that if $s$ edges have been found, then at least $s(k-1)+1$
vertices are explored. Observe that when the stack $S$ is empty there are only explored and neutral vertices.
Further, if $s'$ vertices are explored then Algorithm~2 must have made (at least) $s'\tbinom{n-s'}{k-1}$ 
queries. Further observe that this function is increasing for $s'\le \tfrac{n}{k+2}-1$. We estimate how many edges have been queried  at time $t\le \alpha n^k$. This number is a.a.s.\ at least:
\begin{align*}
& (pt-\alpha^2 n)(k-1)\binom{n-(pt-\alpha^2 n)(k-1)}{k-1}\\
\ge \; \; & \frac{pt-\alpha^2 n}{(k-2)!}(n-pt(k-1))^{k-1}\\
\ge \; \; & \frac{pt-\alpha^2 n}{(k-2)!}n^{k-1}(1-(1+\eps)(k-1)!\alpha )^{k-1}\\
\ge \; \; & (1+\eps/2)t (1-\eps/8) > t
\end{align*}
for $\alpha\le \tfrac{\eps}{8k!}$. However, this 
is a contradiction since we assumed that only $t$ queries have been made so far. Therefore,
for large enough $n$, the stack remains nonempty between queries $\tfrac{\alpha}{2}n^k $ and $\alpha n^k$.
\end{proof}
\begin{remark}\label{rem:vertexpath}
  Similarly to the results in~\cite{KrivSud13} we can show that the large component contains a loose path of length $\Omega_k(\eps^2 n)$ in $H^k(n,p)$
  for $p=(1+\eps)\tfrac{(k-2)!}{n^{k-1}}$. Roughly speaking, the argument is as follows:
We have already shown that the stack of active vertices does not become empty between times $\alpha n^k/2$ and $\alpha n^k$. On the other hand, if the set of active vertices is small enough ($\Theta (\eps^2 n)$ will do), then this will not affect the previous calculations significantly. We can therefore deduce that the set of active vertices never becomes smaller than $\Theta (\eps^2 n)$ in this time interval. But because we are exploring via a depth-first search process, the set of active vertices automatically lies in a loose path (possibly with some explored vertices to complete the edges). 
\end{remark}

\subsection{$j$-tuple-connectivity}
Our aim in this section is to prove Theorem~\ref{thm:pair_conn}. For the remainder of this section we therefore fix $\delta$ and $\eps$ as in Theorem~\ref{thm:pair_conn}. We will also assume that $n\ge n_0$ for some sufficiently large constant $n_0$ which we do not determine explicitly (but which is implicitly dependent on $k,j$ and $\delta$). Let $\alpha = \alpha (n)$ satisfy
\[
\frac{\eps}{32 k! 2^j C} \ge \alpha =\omega(n^{\delta-1}+n^{-j/3}),
\]
where $C$ is a constant depending only on $k,j$ which we determine implicitly later. (We note that for the purposes of this paper, setting $\alpha = \frac{\eps}{32 k! 2^j C}$ would be sufficient. However, in~\cite{CoKaKo14} we will need to quote Lemma~\ref{lem:maxdeg} for a wider range of $\alpha$.)
We first give an outline of the main ideas of the proof.

\subsubsection{Proof sketch of Theorem~\ref{thm:pair_conn}}
The hypergraph $\cH$ which we consider has vertex set $\tbinom{[n]}{j}$ and any $\binom{k}{j}$ $j$-sets, whose union is a $k$-set, form an edge in $\cH$ (cf.\ definition of $\cH$ in Section~\ref{sec:coupling}). 
As explained in Section~\ref{sec:coupling}, the random $\binom{k}{j}$-uniform hypergraph $\cH_p$ we consider is in one-to-one with $H^k(n,p)$. We perform
the same algorithm, Algorithm~2 as described above, i.e.\ only when all vertices but one are 
neutral do we query an edge in $\cH$. Similar to 
 Theorem~\ref{thm:vertex_one} we shall estimate the number of queries made given that 
the stack $S$ is emptied between $\tfrac{\alpha}{2}n^k$ and $\alpha n^k$ queries. 

This time however, we need to take account of the fact that (since $\cH$ 
is clearly \emph{not the complete} hypergraph) not every explored vertex in $\cH$ forms an already queried edge with 
any $\binom{k}{j}-1$ neutral vertices in $\cH$. This is because vertices of $\cH$
 are $j$-element subsets of $[n]$ and edges correspond to only those $\binom{k}{j}$ $j$-sets whose union gives a $k$-element set. 
 Therefore, we will need to keep track of the already discovered $j$-sets of $H^k(n,p)$. 
More precisely, let $G_j$ be the $j$-uniform hypergraph on vertex set $[n]$ whose edges are the $j$-sets which have been discovered by Algorithm~2 up to time $\alpha n^k$ (recall that the $j$-sets are vertices in Algorithm~2). 
We need to bound the degrees of sets of vertices in $G_j$.
 Suppose for the moment that we are able to show Lemma~\ref{lem:maxdeg} below, stating 
 that a.a.s.\ $G_j$ has small maximum degrees depending on $\alpha$. Then from each $j$-set we have made at least $\binom{n-j}{k-j}(1-f(\alpha))$ queries, where $f$ is some function tending to $0$ as $\alpha \rightarrow 0$. Furthermore, we know that we have found approximately $pt$ edges, from each of which we discovered $ \binom{k}{j}-1$ new $j$-sets. Thus the number of queries is at least
\[
pt \left(\binom{k}{j}-1\right)\frac{n^{k-j}}{(k-j)!}(1-f(\alpha))  = (1+\eps)(1-f(\alpha))t >t
\]
for $\alpha$ sufficiently small compared to $\eps$. But this is a contradiction since at time $t$ we have made exactly $t$ queries. This argument will be given in more detail at the end of this section.

\subsubsection{Bounding the maximum degree of $G_j$}
Let $G_j(t)$ denote the $j$-uniform hypergraph on vertex set $[n]$ whose edges are the discovered $j$-sets at time $t$ (so $G_j=G_j(\alpha n^k)$). For each $1\leq \ell <j$, let $\Delta_{\ell}(G_j(t))$ denote the maximum $\ell$-degree of this hypergraph (i.e. the maximum over all $\ell$-sets of the number of edges of $G_j(t)$ containing this $\ell$-set). For convenience, we sometimes use $\Delta_{0}(G_j(t))$ to denote the number of edges in $G_j(t)$ (i.e. the natural generalisation for $\ell=0$). The aim of this section is to prove that a.a.s.\ $G_j(\alpha n^k)$ does not have too large maximum $\ell$-degree for any $0\leq \ell \leq j-1$.

In fact we prove a slightly stronger statement which also applies to a breadth-first search process. We first define two new breadth-first search algorithms:

\begin{itemize}
\item BFS1 is the breadth-first search analogue of Algorithm 1; any edge containing a neutral vertex (which corresponds to a neutral $j$-set) may be queried. Formally, we change line~10 in the algorithm to ``Let $x$ be the bottom vertex of $S$;''.
\item BFS2 is the breadth-first search analogue of Algorithm 2; only edges containing $\binom{k}{j}-1$ neutral vertices (which correspond to neutral $j$-sets) may be queried.
\end{itemize}

We analyse the maximum degrees given by each of the algorithms (Algorithm~1, Algorithm~2, BFS1 and BFS2). Since we will never use specific information about which algorithm we are considering, we will go through all the proofs together and simply refer to the ``search algorithm'', which may be any one of these four. We still use $G_j(t)$ to refer to the hypergraph that has been found by time $t$ using any one of the algorithms.

\begin{lemma}[Bounded degree lemma]\label{lem:maxdeg}  
For some constant $C$, using any one of Algorithm~1, Algorithm~2, BFS1 or BFS2, with probability at least $1-\exp (-n^{\delta/2})$,
\[
\Delta_{\ell}(G_j(\alpha n^k)) \leq C \alpha n^{j-\ell}
\]
for all $0\leq \ell \leq j-1$.
\end{lemma}
Since we will be considering the structure of $G_j(t)$, from now on we will think of the exploration process as one on $j$-sets in $H^k(n,p)$, rather than on vertices in $\cH_p$ (there is of course a natural correspondence between the two).

In fact, we will prove that $\Dfinaltotal \leq C_\ell \alpha n^{j-\ell}$ for each $\ell$, for constants $C_\ell$ which we will determine later, and then we may set $C:= \max_\ell\{C_\ell\}$. Note that by a simple application of the Chernoff bound, the lemma is true for $\ell=0$ if $C_0\ge 2\frac{(k-j)!}{\binom{k}{j}-1}$. For $\ell \ge 1$, we pick an $\ell$-set $L$ and note that there are three ways in which the degree of $L$ in $G_j(t)$ may grow as $t$ increases during the search process:

\begin{enumerate}
\item A {\bf new start} at $L$ occurs when there are no active $j$-sets (all discovered $j$-sets have been explored) and the
 search algorithm picks a new $j$-set from which to start. If this $j$-set contains $L$, then the degree of $L$ in $G_j(t)$ has grown by one.
 Recall that the search algorithm chooses a $j$-set uniformly at random among all neutral $j$-sets, cf.~Algorithm~\ref{alg:DFS}.
\item A {\bf jump} to $L$ occurs when the search process queries a $k$-set $K$ containing $L$ from a $j$-set $J$ not containing $L$ (though possibly intersecting $L$) and the edge $K$ is present. Then for each $A \in \binom{K\backslash L}{j-\ell}$, the $j$-set $A\cup L$ becomes active (if it wasn't already) and the degree of $L$ in $G_j(t)$ grows by at most $\binom{k-\ell}{j-\ell}$ (this is exact for Algorithm~2 or BFS2).
\item From an active $j$-set $J$ containing $L$ we may query a $k$-set $K$ also containing $L$. If this forms an edge then for each $A \in \binom{K\backslash L}{j-\ell}$, the $j$-set $A\cup L$ becomes active (if it wasn't already) and the degree of $L$ in $G_j(t)$ grows by  at most $\binom{k-\ell}{j-\ell}-1$ (this is exact for Algorithm~2 or BFS2). We call this a {\bf branching} at $L$.
\end{enumerate}

We will bound the contributions to the $\ell$-degree $\dtotal$ (which is defined as $|\{J\in E(G_j(t))\colon J\supseteq L\}|$) made by each of these possibilities individually. However we must take care to avoid a circular argument, since the bounds are interdependent.

Let $\Etotal$ be the event that $\Dtotal \leq C_\ell \alpha n^{j-\ell}$ for all $0\leq \ell <j$. We aim to show that with high probability $E{(\alpha n^k)}$ holds, which we do by showing that with high probability, $E{(t-1)} \Rightarrow E{(t)}$ for every $t\leq \alpha n^k$. More precisely, we will first prove some probabilistic lemmas, saying that with high probability, various very likely events will hold throughout the search process. The second part of the proof will be deterministic, showing that conditioned on these good events, $E(t-1) \Rightarrow E(t)$ for any $t \leq \alpha n^k$, and since $E(0)$ automatically holds, by induction $E(\alpha n^k)$ holds.

\subsubsection{Probabilistic Lemmas}
Let us first consider where the new starts are made. Since we select the $j$-set for our new start uniformly at random (it corresponds to choosing a new vertex of $\cH$, and the ordering of $V(\cH)$ was chosen randomly), we expect the new starts to be, in some sense, evenly distributed. The next lemma makes this more precise.

Set $m=2\alpha (k-j)!n^j$. Let $\goodstarts{x}$ be the event that for every $1\le \ell \le j-1$, every 
$\ell$-set is contained in at most $\max \left\{\tfrac{4m j!}{(j-\ell)! n^\ell},n^{\delta}\right\}$ many $j$-sets that were chosen 
to be a new start during the first $x$ new starts. 

Let $\allgoodstarts$ be the intersection of the events $\goodstarts{x}$ over all $x \le m$ and the event $\left\{\sum_{i=1}^{\alpha n^k}X_i \le 2p\alpha n^k\right\}$.

\begin{lemma}\label{lem:goodstarts}
$\Pr (\allgoodstarts) \ge 1-\exp (n^{-\delta/2})$.
\end{lemma}

\begin{proof}
By the Chernoff bound (Theorem~\ref{thm:chernoff}) we have
\[
\Pr \left(\sum_{i=1}^{\alpha n^k}X_i\ge 2p\alpha n^k\right) \le \exp \left( -\frac{p\alpha n^k}{3p\alpha n^k}\right) = \exp (-\Theta (\alpha n^j)) \le \exp(-n^{2/3}) .
\]
Thus we may assume that we have discovered at most $2p\alpha n^k = O(\alpha n^j)$ edges so far, 
and therefore the number of $j$-sets which are discovered is at most $m + \left(\binom{k}{j}-1\right)2p\alpha n^k = O(\alpha n^j)$. 
Thus whenever we made a new start so far, we always had at least $\tfrac{1}{2}\binom{n}{j}$ $j$-sets available to choose from, and so the probability of picking any one of these was certainly at most $2/\binom{n}{j}$.

Now given any $\ell$-set $L$, the number of $j$-sets in which $L$ lies is less than $\binom{n}{j-\ell}$. 
Therefore the number of new starts at a $j$-set containing $L$ has distribution dominated by $\Bi \left(m, 2\binom{n}{j-\ell}/\binom{n}{j}\right)$, which in turn is dominated 
by the binomial distribution $\Bi \left( \max\{m,n^{\ell+\delta/2}\} , \frac{3j!}{(j-\ell)!n^\ell}\right)$.

By the Chernoff bound, the probability that this is greater than $\max\{\tfrac{4m j!}{(j-\ell)! n^\ell},n^{\delta}\}$ is at most $\exp(-n^{2\delta/3})$, and a union bound over all $x$, $\ell$ and $L$ gives the lemma.
\end{proof}

We next prove an auxiliary lemma, which states that we may ``pick out'' certain (random) subsequences of queries and treat them as an interval in the search process. Recall that our sequence of queries gives a sequence of independent Bernoulli random variables $X_1,X_2,\ldots,X_{\binom{n}{k}}$. We will be considering a random subsequence $t_1,t_2,\ldots,t_s$ from $[\binom{n}{k}]$. We say ``$t_i$ is determined by the values of $X_1,\ldots,X_{t_i-1}$'' to mean the following: For any $j$, whether the event $\{t_i=j\}$ holds is determined by the values of $X_1,\ldots,X_{j-1}$. In particular this means that $t_i$ is chosen before $X_{t_i}$ is revealed.

\begin{lemma}\label{lem:subsequence}
Let $S =(t_1,t_2,\ldots,t_s)$ be a (random, ordered) index set chosen according to some criterion such that
\begin{itemize}
\item $t_i$ is determined by the values of $X_1,\ldots,X_{t_i-1}$;
\item with probability $1$ we have $1\le t_1 < t_2 < \ldots < t_s \le \binom{n}{k}$.
\end{itemize}
Then $(X_{t_1},\ldots,X_{t_s}) \sim (Y_1,\ldots,Y_s)$, where $Y_1,\ldots,Y_s$ are independent Be$(p)$ variables. In particular, we may apply a Chernoff bound to $\sum_{i \in S}X_i$.
\end{lemma}

This simple lemma may be folklore, but since we cannot find it in the literature, for completeness we present a proof here.

\begin{proof}
Let $a_1,\ldots,a_s$ be any $\{0,1\}$-sequence of length $s$. For ease of notation, for each $i=1,\ldots,s$ we define $\mathbf{X}^{(i)}:= (X_{t_1},\ldots,X_{t_{i}})$ and $\mathbf{a}^{(i)}:=(a_1,\ldots,a_{i})$. Then we have
\begin{equation}\label{eq:subs1}
\Pr\left(\mathbf{X}^{(s)}=\mathbf{a}^{(s)}\right) =\prod_{i=1}^s \Pr\left(X_{t_i}=a_i \mid \mathbf{X}^{(i-1)}=\mathbf{a}^{(i-1)}\right).
\end{equation}
(Note that for the term $i=1$ in the product, the conditioning is empty.)
Furthermore for any $i$
\begin{align}\label{eq:subs2}
& \Pr \left(X_{t_i}=a_i \mid \mathbf{X}^{(i-1)}=\mathbf{a}^{(i-1)}\right) \nonumber\\
& = \sum_{t=1}^{\binom{n}{k}} \Pr\left(t_i=t \mid \mathbf{X}^{(i-1)}= \mathbf{a}^{(i-1)}\right)  \Pr\left(X_{t_i}=a_i \mid \mathbf{X}^{(i-1)}=\mathbf{a}^{(i-1)} \wedge t_i=t \right).
\end{align}
Now for any choice of $t$ and $\mathbf{a}^{(i-1)}$, let $B=B(t,\mathbf{a}^{(i-1)})$ be the set of all $\{0,1\}$-sequences $\mathbf{b}=(b_1,\ldots,b_{t-1})$ with the property that if $(X_1,\ldots,X_{t-1})=\mathbf{b}$, then $\mathbf{X}^{(i-1)}=\mathbf{a}^{(i-1)}$ and $t_i=t$. Note that this is well-defined since the value of each $t_j$ is uniquely determined by the results of the previous queries.

Then we have 
\begin{align*}
& \Pr\left( \mathbf{X}^{(i-1)}=\mathbf{a}^{(i-1)} \wedge t_i=t \wedge X_{t_i}=a_i \right)\\
= & \sum_{\mathbf{b}\in B(t,\mathbf{a}^{(i-1)})} \Pr \left((X_1,\ldots,X_{t-1})= \mathbf{b} \wedge X_{t_i}=a_i \right)\\
= & \sum_{\mathbf{b}\in B(t,\mathbf{a}^{(i-1)})}  \Pr\left(X_{t}=a_i \mid (X_1,\ldots,X_{t-1})= \mathbf{b}\right) \Pr((X_1,\ldots,X_{t-1})= \mathbf{b})\\
= & \Pr(X_t=a_i)\sum_{\mathbf{b}\in B(t,\mathbf{a}^{(i-1)})} \Pr((X_1,\ldots,X_{t-1})= \mathbf{b})\\
= & \Pr\left(\mbox{Be}(p)=a_i\right) \Pr\left(\mathbf{X}^{(i-1)}=\mathbf{a}^{(i-1)} \wedge t_i=t \right)
\end{align*}
where for the third equality we used the fact that the $X_j$ are all independent. This gives
\[
\Pr\left(X_{t_i}=a_i \mid \mathbf{X}^{(i-1)}=\mathbf{a}^{(i-1)} \wedge t_i=t\right) = \Pr\left(\mbox{Be}(p)=a_i\right)
\]
which we substitute into~\eqref{eq:subs2} to obtain 
\begin{align*}
\Pr\left(X_{t_i}=a_i \mid \mathbf{X}^{(i-1)}=\mathbf{a}^{(i-1)}\right)
& = \sum_{t=1}^{\binom{n}{k}} \Pr\left(t_i=t \mid \mathbf{X}^{(i-1)}=\mathbf{a}^{(i-1)}\right) \Pr\left(\mbox{Be}(p)=a_i\right)\\
& = \Pr\left(\mbox{Be}(p)=a_i\right)
\end{align*}
which in turn we substitute into~\eqref{eq:subs1} to obtain
\[
\Pr\left(\mathbf{X}^{(s)}=\mathbf{a}^{(s)}\right) =\prod_{i=1}^s \Pr\left(\mbox{Be}(p)=a_i\right).
\]
Since $\mathbf{a}^{(s)}$ was arbitrary, the lemma follows.
\end{proof}

We will apply Lemma~\ref{lem:subsequence} to prove two further probabilistic lemmas.

For any $x \in \Nat$, any $1 \le \ell \le j-1$ and any $\ell$-set $L$, let $S(x,L)$ be the set of the first $x$ times at which we make a query which could result in a jump to $L$.

Let $\goodjumps{L}{x}$ be the event that these queries result in at most $2px$ edges (i.e. $\sum_{i \in S(x,L)}X_i \le 2px$). 
Further let $\allgoodjumps$ be the the intersection of all the events $\goodjumps{L}{x}$ over all choices of $\ell, L$ and $x\ge n^{k-j+\delta}$.

\begin{lemma}\label{lem:goodjumps}
For any $n^{k-j+\delta} \le x \in \Nat$, for any $1 \le \ell \le j-1$ and any $\ell$-set $L$, we have $\Pr \left(\goodjumps{L}{x}\right) \ge 1-\exp \left(-n^{2\delta/3}\right)$.\\
Furthermore, $\Pr \left(\allgoodjumps \right) \ge 1-\exp \left(-n^{\delta/2}\right)$.
\end{lemma}

\begin{proof}
 We  apply Lemma~\ref{lem:subsequence} to bound the number of jumps to $L$ within $S(x,L)$. Thus
\begin{align*}
\Pr \left(\goodjumps{L}{x}\right) & \ge 1- \exp \left(-\frac{(px)^2}{3px}\right)\\
& \ge 1- \exp \left(-\frac{n^{\delta}(k-j)!}{3\left(\binom{k}{j}-1\right)}\right)\\
& \ge 1-\exp \left(-n^{2\delta/3}\right).
\end{align*}
For the last statement, we take a union bound over all $\sum_{\ell=1}^{j-1}\binom{n}{\ell} \le n^j$ possible choices of $\ell$ and all choices of $x$ (observing that $x$ is certainly at most $\binom{n}{j} \le n^j$). We therefore obtain
\[
\Pr \left(\allgoodjumps \right) \ge 1-n^{2j}\exp \left(-n^{2\delta/3}\right) \ge 1-\exp \left(-n^{\delta/2}\right)
\]
as required.
\end{proof}

We now aim to prove something similar for the number of branchings at a set $L$ of size $\ell$. Fix $L$ and consider a neighbourhood branching process at $L$. More precisely, given a $j$-set $J$ containing $L$, we make a number of 
queries in the search process and whenever we discover an edge, at most  further $\binom{k-\ell}{j-\ell}-1$ $j$-sets containing $L$ become active (these are considered children of the original $j$-set). For an upper bound we assume exactly $\binom{k-\ell}{j-\ell}-1$ $j$-sets become active.

By deleting $L$ from each of the sets we consider, we may view this as a search process on $(j-\ell)$-sets starting at $J\setminus L$ in a $(k-\ell)$-uniform hypergraph. This may not correspond to a simple time interval 
in the branching process, but we pick out only those queries which are made from a $j$-set containing $L$ (this is permissible by Lemma~\ref{lem:subsequence}). 
The hypergraph in which this search process takes place has $n-\ell$ vertices, but for an upper bound we replace this by $n$. Furthermore, we ignore the fact that some $j$-sets 
may already have been discovered some other way, and are therefore not neutral within this search process. If we further assume that from any $(j-\ell)$-set in the process we may still query $\binom{n}{k-j}$ 
many $(k-j)$-sets, (effectively ignoring the fact that we may have seen some before), then we may consider the process no longer as a hypergraph process, but as an abstract branching process in which the number of children has distribution $r \cdot \Bi \left(\binom{n}{k-j},p\right)$, where $r=r(k,j,\ell)= \binom{k-\ell}{j-\ell}-1$. (By the notation $a\cdot X$, for a real number $a$ and real-valued probability distribution $X$, we mean the probability distribution given by $\Pr(a\cdot X = ai)=\Pr(X=i)$ for any real number $i$.) We first aim to replace this by a distribution which is easier to analyse.

For a probability distribution $P$, let $T_P$ be the tree of a branching process starting at a single vertex in which each vertex has number of children with distribution $P$ independently, and let $\tau_P:=|T_P|$. Thus $\tau_P$ defines a probability distribution on $\mathbb{N}$.

\begin{lemma}\label{lem:distbound}
For any $m,r,s \in \Nat$ and $1/r \ge q \in \mathbb{R}$,
$$\Pr (\tau_{r\cdot \Bi (m,q)}\le rs) \ge \Pr (\tau_{\Bi (rm,q)}\le s).$$
\end{lemma}

\begin{proof}
Consider $T_{r\cdot \Bi (m,q)}$ as a process in which each vertex has $\Bi(m,q)$ clusters of children, each cluster containing $r$ vertices. We consider this as a branching process of clusters. Each cluster has $r$ vertices in it, each of which has $\Bi(m,q)$ cluster-children independently. (The exception is the root of the original branching process, which becomes a cluster on its own.) Thus the number of cluster-children of a cluster has distribution $\Bi(rm,q)$ (except for the root cluster, which has $\Bi(m,q)$ cluster-children). Thus the cluster branching process is dominated by the $\Bi(rm,q)$ branching process. Since each cluster contains $r$ vertices, the lemma follows.
\end{proof}

Let us define, for each $1\le \ell \le j-1$,
\[
c_\ell:=\frac{1}{2}+\frac{1}{2}\frac{\binom{k-\ell}{j-\ell}-1}{\binom{k}{j}-1} < 1
\]
and observe that
\[
\max_{1\le \ell \le j-1}c_\ell = c_1.
\]
Let $\brconst = \brconst(k,j):= \frac{256}{(1-c_1)^4}$.

For any $n^{\delta} \le x \in \Nat$, for any $1 \le \ell \le j-1$ and for any $\ell$-set $L$, let $\goodbranchings{L}{x}$ be the event that the first $x$ neighbourhood branching processes started at $L$ result in at most $\brconst x$ branchings.

Let $\allgoodbranchings$ be the the intersection of all the events $\goodbranchings{L}{x}$ over all choices of $L$ and $x \ge n^\delta$.

\begin{lemma}\label{lem:goodbranchings}
For any $n^\delta \le x \in \Nat$, for any $1 \le \ell \le j-1$ and for any $\ell$-set $L$, with probability at least $1-\exp(-x) \ge 1-\exp(-n^\delta)$, the event $\goodbranchings{L}{x}$ holds. Furthermore, with probability at least $1-\exp(-n^{\delta/2})$, the event $\allgoodbranchings$ holds.
\end{lemma}

\begin{proof}
By the arguments above, for an upper bound we may replace each neighbourhood branching process by the branching process $T_{\Bi\left(r\binom{n}{k-j},p\right)}$. 
Note that the expected number of children is at most  $r\binom{n-j}{k-j}p \le (1+\eps)\frac{\binom{k-\ell}{j-\ell}-1}{\binom{k}{j}-1} \le c_\ell \le c_1 < 1$ for $\eps$ small enough. Thus we may consider $T_{\Bi(N,c_1/N)}$, 
where $N:=\left(\binom{k-\ell}{j-\ell}-1\right)\binom{n-j}{k-j}$.

Therefore let $\tau_1,\ldots,t_x$ be independent identically distributed random variables, where each $\tau_i \sim \tau_{\Bi (N,c_1/N)}$ is the size of such a binomial branching process. Furthermore let $s_x:= \sum_{i=1}^x \tau_i$. Our aim is to show that $s_x \le \brconst x$ with probability at least $1-\exp(-x)$.

Let us therefore consider $\Pr(s_x \ge \brconst x)$. Since the $\tau_i$ are independent, we have
\begin{equation}\label{eq:brprob1}
\Pr(s_x \ge \brconst x) \le \sum_{k_1 + \ldots k_x = \brconst x} \; \prod_{i=1}^x \Pr (\tau_i \ge k_i).
\end{equation}
Furthermore we may couple each branching process ${T}_i$ with an infinite $\{0,1\}$-sequence $Y_1,Y_2,\ldots$ of independent Bernoulli$(c_1/N)$ variables. More precisely, we consider ${T}_i$ to be a subtree of the infinite rooted $N$-ary tree in which each edge is present with probability $c_1/N$ independently. Then ${T}_i$ is the subtree containing the root, and $Y_1,Y_2,\ldots$ represent the queries at each edge of this infinite tree according to a search process -- either depth- or breadth-first search as appropriate. (In fact, the sequence in general need not be infinite, but if we have finished exploring the tree we may consider any remaining variables as dummy variables.)

Now let us observe that in order for the tree ${T}_i$ to have size $k_i$, the $k_i$-th vertex must be found when we discover the $(k_i-1)$-th edge, and up to this point we have made at most $(k_i-1)N$ queries. 
Thus using the Chernoff bound (Theorem~\ref{thm:chernoff}) we have
\begin{align}\label{eq:brprob2}
\Pr (\tau_i \ge k_i) & \le \Pr \left( \sum_{j=1}^{(k_i-1)N} Y_j \ge k_i-1 \right) \nonumber\\
& \le  \exp \left( -\frac{(1-c_1)^2(k_i-1)^2}{2\left(c_1(k_i-1)+(1-c_1)(k_i-1)/3\right)} \right) \nonumber\\
& \le \exp \left( -(k_i-1) \frac{(1-c_1)^2}{2}\right).
\end{align}
Thus substituting~\eqref{eq:brprob2} into~\eqref{eq:brprob1} we have
\begin{align}\label{eq:brprob3}
\Pr (s_x \ge \brconst x) & \le \sum_{k_1+\ldots k_x=\brconst x} \prod_{i=1}^x \exp \left( -(k_i-1) \frac{(1-c_1)^2}{2}\right) \nonumber \\
& = \sum_{k_1+\ldots k_x=\brconst x} \exp \left( (1-\brconst )x \frac{(1-c_1)^2}{2}\right) \nonumber \\
& = \binom{\brconst x-1}{x-1} \exp \left( (1-\brconst )x \frac{(1-c_1)^2}{2}\right).
\end{align}
Furthermore,
\begin{align}\label{eq:brprob4}
\binom{\brconst x-1}{x-1} & \le \left( \frac{e(\brconst  x-1)}{x-1} \right)^{x-1} \nonumber \\
& = \exp \left( (x-1)\left(1+\log \left(\frac{\brconst x-1}{x-1}\right)\right) \right) \nonumber \\
& \le \exp \left( 2x \log \brconst \right),
\end{align}
where the last inequality certainly holds if $\brconst \ge e$.
Substituting ~\eqref{eq:brprob4} into~\eqref{eq:brprob3}, we obtain
\begin{align}\label{eq:brprob5}
\Pr (s_x \ge \brconst x) & \le \exp \left(2x\log \brconst + (1-\brconst )x \frac{(1-c_1)^2}{2}\right) \nonumber \\
& \le \exp \left(- \frac{(1-c_1)^2\brconst x}{4} \right),
\end{align}
where the last line holds for $\brconst $ large enough (dependent on $c_1$). In particular, it certainly holds provided
\[\brconst \ge \frac{16 \log \brconst}{(1-c_1)^2}\]
which in turn holds provided
\[
\brconst \ge \frac{256}{(1-c_1)^4}.
\]
Finally, we observe that for such $\brconst$, we also have $ \frac{(1-c_1)^2\brconst}{4} \ge 1$, and therefore by~\eqref{eq:brprob5}
\[
\Pr (s_x \ge \brconst x)  \le \exp (-x) \le \exp(-n^{\delta})
\]
as required by the lemma.

This proves the first part of the lemma, and for the second part we simply take a union bound over all choices of $\ell, L$ and $x$ (of which there are certainly at most $n^{2j}$ in total).
\end{proof}

Finally let
\[
\allgoodevents := \allgoodstarts \wedge \allgoodjumps \wedge \allgoodbranchings.
\]
The following is an immediate corollary of Lemmas~\ref{lem:goodstarts}, \ref{lem:goodjumps} and~\ref{lem:goodbranchings}.

\begin{cor}
$\Pr (\allgoodevents) = 1- 3\exp \left(-n^{\delta/2} \right)$.\endproof
\end{cor}

\subsubsection{Inductive Proof}
Recall that $\Etotal$ is the event that $\Dtotal \leq C_\ell \alpha n^{j-\ell}$ for all $0\leq \ell <j$.
In this section we will show that $\allgoodevents \Rightarrow E(\alpha n^k)$. More precisely, we prove that $\allgoodevents\Rightarrow E(t)$ for all $t\le \alpha n^k$ by induction on $t$. The base case is trivial, since $E(0)$ holds with probability $1$.

\begin{itemize}
\setlength\itemsep{6pt}
\item Let $\dstart$ be the number of new starts at $L$ by time $t$ and let\\
$\Dstart :=\max \dstart$, where the maximum is over all sets $L$ of size $\ell$.
\item Let $\djump$ be the number of jumps to $L$ by time $t$ and let\\
$\Djump :=\max \djump$, where the maximum is over all sets $L$ of size $\ell$.
\item Let $\dbranch$ be the number of branchings at $L$ up to time $t$ and let\\
$\Dbranch :=\max \dbranch$, where the maximum is over all sets $L$ of size $\ell$.
\end{itemize}
Let $\hat{C}_0:=2$, $C_0^*:=2$ and recursively define 
\begin{align*}
C_\ell & :=\max \left\{ \hat{C}_\ell+C^*_\ell+\tfrac{8j!(k-j)!}{(j-\ell)!},C_{\ell-1} \right\};\\
\hat{C}_{\ell+1} & :=\max \left\{2^{\ell+2}\tfrac{(k-j)!}{\binom{k}{j}-1}C_{\ell},8\right\};\\
C^*_{\ell+1} &:=2k! \hat{C}_{\ell+1} \brconst
\end{align*}
for $\ell\ge 0$, where $\brconst$ is the constant from Lemma~\ref{lem:goodbranchings}. (For the sake of the definition of $C_0$, we adopt the convention that $C_{-1}=0$.)

\begin{itemize}
\setlength\itemsep{6pt}
\item Let $\Ecomp$ be the event that for each $0\leq \ell <j$, $\Dstart \leq \tfrac{8j!(k-j)!}{(j-\ell)!} \alpha n^{j-\ell}$.
\item Let $\Ejump$ be the event that for each $0\leq \ell <j$, $\Djump \leq \hat{C}_\ell \alpha n^{j-\ell}$.
\item Let $\Ebranch$ be the event that for each $0\leq \ell <j$, $\Dbranch \leq C_\ell^* \alpha n^{j-\ell}$.
\end{itemize}

\vspace{0.15cm}

\noindent Note that since $C_\ell \ge \hat{C}_\ell + C_\ell^* +\tfrac{8j!(k-j)!}{(j-\ell)!}$ we have $\Ecomp \wedge \Ejump \wedge \Ebranch \Rightarrow \Etotal$. We aim to show that, conditioned on $\allgoodevents$, none of these can be the first to fail before time $\alpha n^k$. However, we must be careful with the time steps since it may be that two of these events become false simultaneously.

\begin{lemma}\label{lem:newstarts}
$\allgoodstarts \wedge \Etotal \Rightarrow \allgoodstarts \wedge E^{(1)}(t+1)$ for $t \le \alpha n^k$.
\end{lemma}

\begin{proof}
That $\allgoodstarts \wedge \Etotal \Rightarrow \allgoodstarts$ is immediate, so we only need to show that $\allgoodstarts \wedge \Etotal \Rightarrow E^{(1)}(t+1)$. 
Note that by $\Etotal$, we have $\Delta_{\ell}(G_j(t))\le C_\ell \alpha n^{j-\ell}$ for all $0\le \ell\le j-1$. Thus, 
for each $j$-set we made at least
\[
\binom{n-j}{k-j}-\sum_{\ell=0}^{j-1} \binom{j}{\ell} \Delta_{\ell}(G_j(t))\binom{n-2j+\ell}{k-2j+\ell}=(1-O(\alpha))\binom{n}{k-j}
\]
queries. 
Thus, the number of new starts we have made is certainly at most
\[
\frac{\alpha n^k}{(1-O(\alpha))\binom{n}{k-j}} \le 2\alpha (k-j)!n^j.
\]
For an upper bound, we will assume that we have made exactly this many. Then by $\allgoodstarts$, for any $\ell$-set $L$ we have made at most $\frac{8j!(k-j)!}{(j-\ell)!}\alpha n^{j-\ell}$ new starts at $L$, as required.
\end{proof}

\begin{lemma}\label{lem:jumps}
$\allgoodjumps \wedge \Etotal \Rightarrow \allgoodjumps \wedge E^{(2)}(t+1)$ for $t \le \alpha n^k$.
\end{lemma}

\begin{proof}
Similarly to Lemma~\ref{lem:newstarts}, it is enough to show that $\allgoodjumps \wedge \Etotal \Rightarrow E^{(2)}(t+1)$

Given an $\ell$-set $L$, we consider the number of jumps to $L$ by time $t$. For each $0\leq i <\ell$, the number of queries to $L$ from $j$-sets which intersect $L$ in a set $I$ of $i$ vertices is certainly at most $\Delta_{i}(G_j(t))\leq C_i \alpha n^{j-i}$ (since $E{(t)}$ holds, we can bound the number of $j$-sets which have been active and contain $I$). We have $\binom{\ell}{i}$ such sets $I$, and for each of these, if we are to jump to $L$ we have already chosen $j+\ell-i$ vertices, and therefore have at most $\binom{n}{k-j-\ell +i}$ choices for the remaining vertices. Thus the total number of queries by time $t$ which may have resulted in jumps to $L$ is at most
$$
\sum_{i=0}^{\ell -1} \binom{\ell}{i}C_i \alpha n^{j-i} \binom{n}{k-j-\ell +i} \leq 2^\ell C_{\ell-1} \alpha n^{k-\ell}.
$$ 
Thus by $\allgoodjumps$, the number of jumps to $L$ is at most
\begin{align*}
2\cdot 2^\ell C_{\ell-1} \alpha n^{k-\ell}p & = (1+\eps)2^{\ell+1} C_{\ell-1} \alpha n^{j-\ell}\frac{(k-j)!}{\binom{k}{j}-1}\\
& \leq \hat{C}_\ell \alpha n^{j-\ell}.
\end{align*}
Since $L$ was chosen arbitrarily, this holds for all $L$, and therefore $E^{(2)}(t+1)$ is satisfied, as required.
\end{proof}

\begin{lemma}\label{lem:branchings}
$\allgoodbranchings \wedge \Etotal \Rightarrow \allgoodbranchings \wedge E^{(3)}(t+1)$ for $t \le \alpha n^k$.
\end{lemma}

\begin{proof}
Since we assume that $E{(t)}$ holds, the number of neighbourhood branching processes which we start at a set $L$ is at most the number of new starts at $L$ plus $\binom{k-\ell}{j-\ell}$ times the number of jumps to $L$, or at most
\[
\left(\frac{8j!(k-j)!}{(j-\ell)!} + \binom{k-\ell}{j-\ell}\hat{C}_\ell \right) \alpha n^{j-\ell} \le  2k!\hat{C}_\ell \alpha n^{j-\ell}.
\]

For an upper bound, we will assume that we have exactly $2k! \hat{C}_\ell  \alpha n^{j-\ell} \ge n^{\delta}$ neighbourhood branching processes. Then by $\goodbranchings{L}{2k! \hat{C}_\ell  \alpha n^{j-\ell}}$, the total number of vertices in all of these branching processes is at most $2k! \brconst \hat{C}_\ell \alpha n^{j-\ell}$
as required.

Since $L$ was chosen arbitrarily, this holds for any $L$, and thus $E^{(3)}(t+1)$ holds.
\end{proof}

Now combining Lemmas~\ref{lem:newstarts},~\ref{lem:jumps} and~\ref{lem:branchings}, we have that for $t \le \alpha n^k$
\begin{align*}
\allgoodevents \wedge \Etotal & \Rightarrow \allgoodevents \wedge E^{(1)}(t+1) \wedge E^{(2)}(t+1) \wedge E^{(3)}(t+1)\\
& \Rightarrow \allgoodevents \wedge E(t+1).
\end{align*}
Since $E(0)$ holds trivially, by induction we may deduce that
$\allgoodevents \Rightarrow E(\alpha n^k)$, and therefore
\[
\Pr ( E(\alpha n^k)) \ge \Pr (\allgoodevents) \ge 1- 3\exp \left( -n^{\delta/2} \right)
\]
as required. This completes the proof of Lemma~\ref{lem:maxdeg}.

\subsubsection{Proof of Theorem~\ref{thm:pair_conn}}\label{sec:endofproof}

We now complete the proof by filling in the details of the argument sketched earlier. We now choose $\alpha = \frac{\eps}{32 k! 2^j C}$.

 We assume that the stack $S$ is empty at some time $t\in\{\tfrac{\alpha}{2}n^k,\ldots,\alpha n^k\}$ 
(and thus there is a new while-loop 
 between queries $\tfrac{\alpha}{2}n^k$ and $\alpha n^k$). Thus, at time $t$ we can estimate, 
using Lemma~\ref{lem:strongconcentration}, that a.a.s.\ at least $tp - \alpha^2 n^j$ edges have been found by time $t$. Recall, that since we run Algorithm~2, 
whenever an edge appears, we discover $\binom{k}{j}-1$ new $j$-sets of vertices, so $e(G_j(t)) \ge (\binom{k}{j}-1)(tp - \alpha^2 n)$ a.a.s..

We note that
\begin{align*}
\frac{\alpha^2 n^j}{tp} & \le \frac{2\left(\binom{k}{j}-1\right)\alpha}{(k-j)!} \le k! \alpha \le \eps/4
\end{align*}
and so 
$$\left(\binom{k}{j}-1\right)\left(tp-\alpha^2 n^j \right)\ge \left(\binom{k}{j}-1\right)tp(1-\eps/4).$$
Furthermore, we know from Lemma~\ref{lem:maxdeg} that $\Delta_{\ell}(G_j(t))\le C\alpha n^{j-\ell}$ a.a.s..  Since every 
$k$-subset of $[n]$ which contains exactly one $j$-set which is an edge of $G_j(t)$ and $\binom{k}{j}-1$ not in $G_j(t)$ (the stack $S$ is empty)
must have been queried at this time, we infer that at time $t$ (a.a.s.) at least 
\begin{align*}
& \left(\binom{k}{j}-1\right)tp(1-\eps/4)\left(\binom{n-j}{k-j}-\sum_{\ell=0}^{j-1} \binom{j}{\ell} \Delta_{\ell}(G_j(t))\binom{n-2j+\ell}{k-2j+\ell}\right)\\
> \;\; & t(1+3\eps/5)\frac{(k-j)!}{n^{k-j}} \left(\binom{n-j}{k-j}- 2^j C \alpha n^{j-\ell} n^{k-2j+\ell}\right)\\
> \;\; & (1+\eps/2)(1-2^j (k-j)!C\alpha)t
\end{align*}
queries were made. This is larger than $t$ if $\alpha\le \eps/(2^{j+2}(k-j)!C)$ and therefore we obtain a contradiction (since until this time 
only $t$ queries were made). Thus, between $\alpha n^k/2$ and $\alpha n^k$ the stack remains nonempty, which again implies by Lemma~\ref{lem:strongconcentration} 
that at least $\alpha p n^k/2 - 2\alpha^2 n^j$ edges are in some $j$-tuple-connected component, which therefore contains at least
$$\left(\binom{k}{j}-1\right)\left(\alpha p n^k/2 - 2\alpha^2 n^j \right) = \Omega(\eps n^j)$$
$j$-sets. This completes the proof of Theorem~\ref{thm:pair_conn}

\begin{remark}\label{rem:walk}
As we did for vertex-connectivity, we could modify our calculations to prove that with high probability the set of active $j$-sets does not become small between times $\alpha n^k/2$ and $\alpha n^k$. However, for $j>1$ the set of active $j$-sets does \emph{not} automatically form a $j$-tight path since they could, for example, all contain one vertex. We would however obtain a long $j$-tight walk which is non-repeating in the sense that a $j$-set is only visited once in the walk.
\end{remark}

\begin{remark}
The most difficult part of the proof, the bounded degree lemma (Lemma~\ref{lem:maxdeg}), explicitly allowed the search process to be a breadth-first search rather than a depth-first search. In fact, the rest of the proof would also work equally well for a breadth-first search. The only point at which we actually need a depth-first search process is in Remark~\ref{rem:vertexpath} and Remark~\ref{rem:walk}, where we note that the set of active vertices forms either a path or a $j$-tight walk. The breadth-first search algorithm is used in~\cite{CoKaKo14}.
\end{remark}

\section{Concluding remarks}\label{sec:conclusion}
For $p=(1+\eps)p_{k,j}$, a natural conjecture is that a \emph{unique} largest component of size $\Omega (\eps n^j)$ should exist with high probability for any 
$\eps$ such that $\eps^3 n^j\rightarrow \infty$. In this paper, we have the additional condition that $\eps \gg n^{\delta-1}$ (for some $\delta >0$). For $j=1,2$, this condition is already 
implied by $\eps^3 n^j \rightarrow \infty$, so in these cases our range of $\eps$ is best possible. However, once $j\ge 3$, the condition $\eps \gg n^{\delta-1}$ takes over.

The extra condition arises because of our proof method; in the bounded degree lemma, we wish to show that degrees which we expect to have size $\Theta (\eps n^{j-\ell})$ do not exceed 
their expected size by more than a constant factor (a.a.s.). For this to be plausible, we certainly need $\Theta (\eps n^{j-\ell})$ to be large, which for $\ell = j-1$ leads to the extra condition on $\eps$. If one were to attempt to remove this condition while still using this proof method, presumably some information about the distribution of degrees (which may now be small) would be required.

We have shown here, that the largest component has size $\Omega (\eps n^j)$, which for constant $\eps$ is certainly the correct order of magnitude. In \cite{CoKaKo14}, the asymptotic size of the largest component is determined and its uniqueness (i.e. that all other components are much smaller) proven, although the range of $\eps$ is slightly more restrictive than that allowed here. The argument in that paper makes fundamental use of the bounded degree lemma from this paper. Independently Lu and Peng~\cite{LP14} also claim to have proved the asymptotic size and uniqueness of the largest component, though only for constant $\eps$.

It would also be interesting to know about the structure of the components and in particular whether there is a simple generalisation of the well-known fact that for graphs all small components (i.e. any except the giant component, if it exists) are either trees or unicyclic graphs a.a.s.. For the case $j=1$, results in this direction were obtained in~\cite{KarLuc02,RR06}.

Finally, one could also study the emergence of the \emph{$s$-cores} of a random hypergraph: For $1\leq \ell <k$ we have defined the degree of a set of $\ell$ vertices, and so we have a well-defined notion of minimum $\ell$-degree. We can therefore ask when a.a.s.\ there exists a non-empty subhypergraph of $H^k(n,p)$ with minimum $\ell$-degree at least $s$, which is called the $s$-core. This has already been studied in the case $\ell=1$ by Molloy~\cite{Mo05}, but for other values of $\ell$ this question remains wide open.

\bibliographystyle{amsplain}
\bibliography{CKP14Bibliography}
\end{document}